\documentclass{amsart}
\usepackage{amsmath}
\usepackage{amsthm}
\usepackage{amssymb}
\usepackage{graphicx}
\usepackage{amsfonts}
\usepackage{latexsym}
\usepackage{setspace}
\usepackage{mathdots}
\usepackage{mathrsfs}
\usepackage{comment}
\usepackage{enumerate}
\usepackage{epsfig}
\usepackage{color}
\usepackage[all]{xy}

\usepackage{bm}
\usepackage{tkz-linknodes}

\usepackage[active]{srcltx}

 \numberwithin{equation}{section}
 \theoremstyle{plain}
\newtheorem{Theorem}[equation]{Theorem}
\newtheorem{Lemma}[equation]{Lemma}
\newtheorem{Proposition}[equation]{Proposition}
\newtheorem{Corollary}[equation]{Corollary}
\theoremstyle{remark}

\theoremstyle{definition}

\newtheorem{Example}[equation]{Example}

\allowdisplaybreaks

\def\phi{\varphi}

\renewcommand{\leq}{\leqslant}
\renewcommand{\geq}{\geqslant}

\makeindex
\begin{document}

\title[Weak Parallelogram Constants]{Optimal Weak Parallelogram Constants for $L^p$ Spaces}

\author[Cheng]{Raymond Cheng}
\address{Department of Mathematics and Statistics,
  Old Dominion University,
  Norfolk, VA 23529}
  \email{rcheng@odu.edu}

\author[Mashreghi]{Javad Mashreghi}
\address{D\'epartement de math\'ematiques et de statistique, Universit\'e Laval, Qu\'ebec, QC, Canada, G1V 0A6}
\email{javad.mashreghi@mat.ulaval.ca}

\author[Ross]{William T. Ross}
	\address{Department of Mathematics and Computer Science, University of Richmond, Richmond, VA 23173, USA}
	\email{wross@richmond.edu}



\begin{abstract}
Inspired by Clarkson's inequalities for $L^p$ and continuing work from \cite{CR}, this paper computes the optimal constant $C$ in the weak parallelogram laws
$$ \|f + g \|^r +  C\|f - g\|^r \leq 2^{r-1}\big( \|f\|^r + \|g\|^r  \big), $$
$$  \|f + g \|^r +  C\|f- g \|^r \geq 2^{r-1}\big( \|f\|^r + \|g \|^r  \big)$$
for the $L^p$ spaces, $1 < p < \infty$. 
\end{abstract}

\thanks{The second author was supported by NSERC}
	
\maketitle


\section{Introduction}


Let $C>0$ and $1<r<\infty$.  Following \cite{CR}, we say that
 a Banach space $\mathscr{X}$ with norm denoted by $\|\cdot\|$ satisfies the {\it $r$-lower weak parallelogram law with constant $C$} (in brief, $\mathscr{X}$ is $r$-LWP($C$)), if
\begin{equation}\label{LWPeq}
  \|\mathbf{x} + \mathbf{y} \|^r +  C\|\mathbf{x} - \mathbf{y} \|^r \leq 2^{r-1}\big( \|\mathbf{x}\|^r + \|\mathbf{y} \|^r  \big), \quad \mathbf{x}, \mathbf{y} \in \mathscr{X}.
\end{equation}
Similarly, $\mathscr{X}$ satisfies the {\it $r$-upper weak parallelogram law with constant $C$} (in brief, $\mathscr{X}$ is $r$-UWP($C$)), if
\begin{equation}\label{UWPeq}
  \|\mathbf{x} + \mathbf{y} \|^r +  C\|\mathbf{x} - \mathbf{y} \|^r \geq 2^{r-1}\big( \|\mathbf{x}\|^r + \|\mathbf{y} \|^r  \big), \quad \mathbf{x}, \mathbf{y} \in \mathscr{X}
\end{equation}
Let us refer to $C$ as the {\it weak parallelogram constant}. Substituting $\mathbf{x} = - \mathbf{y}$ in \eqref{LWPeq} and \eqref{UWPeq}, observe that $0 < C \leq 1$ in \eqref{LWPeq} and $C \geq 1$ in \eqref{UWPeq}.

There are many applications arising from the weak parallelogram laws.  They were employed in \cite{CRM} to obtain bounds for the zeros of an analytic function, extending the classical bounds due to Cauchy, Lagrange and others.  In \cite{CR2} these tools from Banach space geometry made possible a novel sort of Beurling-Riesz factorization for functions in the space $\ell^p_A$.   Furthermore, this circle of ideas was useful in the prediction theory of stochastic processes endowed with an $L^p$ structure.  In \cite{CR} an $L^p$ Baxter-type inequality and criteria for regularity were obtained; and in \cite{CR2}, stationary, non-anticipating solutions were found for certain ARMA processes with heavy tails. We will unpack more of the history and further connections in Section \ref{back}.  

By extending Clarkson's inequalities, it was shown in \cite{CR} that when $1<p<\infty$, the Lebesgue space $L^p = L^{p}(X, \Omega, \sigma)$ satisfies the upper and lower weak parallelogram laws for a range values of the parameters $C$ and $r$.  The main result of this paper identifies the {\em optimal} values of the weak parallelogram constants in these inequalities.  That is, for  given $r$ and $p$, we supply the largest value of $C$ for which (\ref{LWPeq}) holds, and the smallest for which (\ref{UWPeq}) holds.
To this end, define 
\begin{equation}
       C_{p,r} := \inf_{0 \leq t < 1} \frac{2^{r-r/p}(1+t^p)^{r/p}-(1+t)^r}{(1-t)^r}.
\end{equation}
We will study $C_{p, r}$ more carefully in Section \ref{Crp}. 
With that, here are the weak parallelogram laws satisfied by $L^p$, in their optimal form.
\begin{Theorem}[Optimal Constants Theorem]\label{MAINT}
If $1 < p \leq 2$ and $q$ is the conjugate index ($1/p + 1/q = 1$),  then $L^p$ is:
  \begin{align}
& \mbox{ $r$-UWP(1)  when   $1<r \leq p$;}\label{wkparaeq1}\\
& \mbox{ $r$-LWP($C_{p,r}$)  when $2\leq r \leq q$; and}\label{wkparaeq2}\\
& \mbox{ $r$-LWP(1)  when  $q\leq r <\infty$.}\label{wkparaeq3}
\end{align}
If $2 \leq p < \infty$, then $L^p$ is:
\begin{align}
& \mbox{ $r$-LWP(1)  when  $p \leq r < \infty$; }\label{wkparaeq4}\\
& \mbox{ $r$-UWP($C_{q,r'}^{-p/q}$)  when  $q \leq r \leq 2$; and}\label{wkparaeq5}\\
& \mbox{ $r$-UWP(1)  when  $1<r \leq q$,}\label{wkparaeq6}
\end{align}
where $1/r + 1/r' = 1$.  The weak parallelogram constants are optimal.
\end{Theorem}

The alert reader probably has noticed that in Theorem \ref{MAINT} we seem to be missing a discussion of upper weak parallelogram inequalities for certain $r$ and $p$. This is not a oversight since, as discussed in \cite[Prop.~3.6]{CR}, $L^p$ satisfies no $r$-UWP law when $r>p$ or $r>2$, and no $r$-LWP law when $r<p$ or $r<2$. 
 
In Proposition \ref{constants-BJ} we discuss relationships between the weak parallelogram constants and the constants 
 of James, and of von Neumann and Jordan, which explore how far a given Banach space is from a Hilbert space.

\section{Background}\label{back}

A complex Hilbert space $\mathscr{H}$ with norm $\|\cdot\|$ satisfies the {\it Parallelogram Law}:  
\begin{equation}\label{paralaw}
     \|\mathbf{x} + \mathbf{y} \|^2 + \|\mathbf{x} - \mathbf{y} \|^2 = 2\big( \|\mathbf{x}\|^2 + \|\mathbf{y} \|^2  \big), \quad \mathbf{x}, \mathbf{y} \in \mathscr{H}.
\end{equation}
If we consider the parallelogram in $\mathscr{H}$ with vertices $\mathbf{0}$, $\mathbf{x}$, $\mathbf{y}$ and $\mathbf{x}+\mathbf{y}$,  this says that the sum of the squares of its diagonals equals the sum of the squares of its four sides.  
Conversely,  a celebrated theorem of Jordan and von Neumann \cite{MR1503247} (see also \cite[Ch.~7]{Day}) says that any Banach space $\mathscr{X}$ that satisfies the Parallelogram Law turns out to be a Hilbert space, where the inner product can be retrieved via the 
Polarization Identity
\begin{equation*}\label{polarizeq}
    \langle \mathbf{x},\mathbf{y} \rangle := \textstyle{\frac{1}{4}}\big( \|\mathbf{x} + \mathbf{y}\|^2
     + i\|\mathbf{x} + i\mathbf{y}\|^2 - \|\mathbf{x} - \mathbf{y}\|^2 -i\|\mathbf{x} -i \mathbf{y}\|^2 \big).
\end{equation*}
In fact, a little thought will show that if (\ref{paralaw}) is weakened to
\begin{equation}\label{paraineqeq}
     \|\mathbf{x} + \mathbf{y} \|^2 + \|\mathbf{x} - \mathbf{y} \|^2 \leq 2\big( \|\mathbf{x}\|^2 + \|\mathbf{y} \|^2  \big), \quad \mathbf{x}, \mathbf{y} \in \mathscr{X}
\end{equation}
then $\mathscr{X}$ is still compelled to be a Hilbert  space.

In light of this situation, it makes sense to explore the degree to which a Banach space departs from the Parallelogram Law, or from the inequality (\ref{paraineqeq}).  The weak parallelogram laws (\ref{LWPeq}) and (\ref{UWPeq}), by allowing an exponent $r$ other than 2, and by the insertion of the weak parallelogram constant $C$, provide an avenue for this type of exploration.   
The Parallelogram Law proper is the special case where both $2$-LWP($1$) and $2$-UWP($1$) hold, and we have already seem that this characterizes Hilbert spaces.  In the context of $L^p$ spaces, two of Clarkson's Inequalities \cite[p.~117]{Car} tell us that
\begin{align*}
     \|f + g \|_{L^p}^p +  \|f - g \|_{L^p}^p &\geq 2^{p-1}\big( \|f\|_{L^p}^p + \|g \|_{L^p}^p  \big), \quad 1 < p \leq 2,\\  
    \|f + g \|_{L^p}^p +  \|f - g \|_{L^p}^p &\leq 2^{p-1}\big( \|f\|_{L^p}^p + \|g \|_{L^p}^p  \big), \quad 2 \leq p < \infty.
\end{align*}
In other words, $L^p$ is $p$-UWP($1$) when $1<p\leq 2$, and is $p$-LWP($1$) when $2 \leq p <\infty$.  The weak parallelogram laws with $r=2$ were first studied by Bynum and Drew \cite{BD} in relation to the sequence space $\ell^p$, and in Bynum \cite{Byn} for general Banach spaces.  Cheng et al.~\cite{CMP} examined the $r=p$ case and derived the corresponding form of the Pythagorean theorem. In \cite{CR} the weak parallelogram laws were studied in their full generality, with connections made to the geometric notions of convexity, smoothness, and orthogonality.  It turns out, for example, that no weak parallelogram laws are satisfied by $L^1$ or $L^{\infty}$, since they fail to be uniformly convex or uniformly smooth \cite[Ch.~11]{Car}.   A duality theorem for the weak parallelogram laws was established in \cite{CH1}, along with further properties of the parameters.  It is also shown that the weak parallelogram laws are preserved under quotients, subspaces and Cartesian products.  This allows us to construct a wide range of weak parallelogram spaces.  

The weak parallelogram laws are a vehicle to describe the geometry of a Banach space, including how much it departs from that of a Hilbert space.  Another approach along these lines stems from the study of certain characteristic constants.  For example, the  von Neumann-Jordan Constant $C_{N\! J}$ \cite{Clax} (see also \cite{KMN}) for a Banach space $\mathscr{X}$ is the smallest value of $C$ for which
\[
     \frac{1}{C} \leq  \frac{\|\mathbf{x}+\mathbf{y}\|^2 + 
   \|\mathbf{x}-\mathbf{y}\|^2}{2(\|\mathbf{x}\|^2+\|\mathbf{y}\|^2)}\leq C
\]
for all $\mathbf{x}$ and $\mathbf{y}\in\mathscr{X}$, not both zero. It is known that $C_{N\! J}$ exists and $1 \leq C_{N\! J} \leq 2$.
The James Constant $J$ \cite{GL} of $\mathscr{X}$ is given by
\[
   J := \sup\{\min(\|\mathbf{x}+\mathbf{y}\|,\|\mathbf{x}-\mathbf{y}\|)\}
\]
as $\mathbf{x}$ and $\mathbf{y}$ vary over the unit sphere of $\mathscr{X}$.
The weak parallelogram laws relate to these constants in the following manner.
\begin{Proposition}\label{constants-BJ}
Let $\mathscr{X}$ be a Banach space.
\begin{enumerate}
\item[(i)]  If $\mathscr{X}$ is $2$-LWP($C$), then $C_{N\! J} \leq 1/C$;
\item[(ii)] If $\mathscr{X}$ is $2$-UWP($C$), then $C_{N\! J} \leq C$;
\item[(iii)] If $\mathscr{X}$ is $r$-LWP($C$), then $J \leq 2/(1+C)^{1/r}$;
\item[(iv)] If $\mathscr{X}$ is $r$-UWP($C$), then $J \leq 2^{1-2/r}(1+C)^{1/r}$.
\end{enumerate}
\end{Proposition}

\begin{proof}
Suppose that $\mathscr{X}$ is $2$-LWP($C$).  Then, as mentioned in the introduction,  $0 < C \leq 1$, and hence
\begin{align*} 
     C\|\mathbf{x} + \mathbf{y} \|^2 +  C\|\mathbf{x} - \mathbf{y} \|^2
     & \leq \|\mathbf{x} + \mathbf{y} \|^2 +  C\|\mathbf{x} - \mathbf{y} \|^2 \\
     & \leq 2\big( \|\mathbf{x}\|^2 + \|\mathbf{y} \|^2  \big).
\end{align*}
This proves that $C_{N\! J} \leq 1/C$, as claimed.

Next, suppose that $\mathscr{X}$ is $2$-UWP($C$).  That is,
\[
       \|\mathbf{x} + \mathbf{y} \|^2 +  C\|\mathbf{x} - \mathbf{y} \|^2 \geq 2\big( \|\mathbf{x}\|^2 + \|\mathbf{y} \|^2  \big)
\]
for all $\mathbf{x}$ and $\mathbf{y} \in \mathscr{X}$.  By writing $\mathbf{X} = \mathbf{x}+\mathbf{y}$ and $\mathbf{Y}=\mathbf{x}-\mathbf{y}$, we can restate this as
\[
        \|\mathbf{X}\|^2 + C\|\mathbf{Y}\|^2 \geq 2\Big( \Big\| \frac{\mathbf{X}+\mathbf{Y}}{2} \Big\|^2 +  \Big\| \frac{\mathbf{X}-\mathbf{Y}}{2} \Big\|^2 \Big),
\]
or, reverting back to lowercase letters,
\[
      \|\mathbf{x} + \mathbf{y} \|^2 +  \|\mathbf{x} - \mathbf{y} \|^2 \leq 2\big( \|\mathbf{x}\|^2 + C\|\mathbf{y} \|^2  \big).
\]
Since, again as in the introduction, $1 \leq C <\infty$ in this case, the right side of this last inequality is bounded by
$
     2C( \|\mathbf{x}\|^2 + \|\mathbf{y} \|^2 ).
$
It follows that $C_{N\! J} \leq C$.

We now assume that $\mathscr{X}$ is $r$-LWP($C$).  Then for any $\mathbf{x}$ and $\mathbf{y} \in \mathscr{X}$,
\begin{align*}
    (1+C) \min \{\|\mathbf{x} + \mathbf{y} \|^r,\|\mathbf{x} - \mathbf{y} \|^r\} & \leq
        \|\mathbf{x} + \mathbf{y} \|^r +  C\|\mathbf{x} - \mathbf{y} \|^r\\
        & \leq 2^{r-1}\big( \|\mathbf{x}\|^r + \|\mathbf{y} \|^r  \big).
\end{align*}
Take the supremum of the initial expression over the unit sphere of $\mathscr{X}$, and extract $r$th roots, to obtain
$
     J \leq 2/(1+C)^{1/r}.
$

Finally if $\mathscr{X}$ is $r$-UWP($C$), then
\begin{align*}
     2 \min \{\|\mathbf{x} + \mathbf{y} \|^r,\|\mathbf{x} - \mathbf{y} \|^r\} & 
       \leq
        \|\mathbf{x} + \mathbf{y} \|^r +  \|\mathbf{x} - \mathbf{y} \|^r\\
        & \leq 2^{r-1}\big( \|\mathbf{x}\|^r + C\|\mathbf{y} \|^r  \big).
\end{align*}
From this we deduce that 
$
     J \leq 2^{1-2/r}(1+C)^{1/r}.
$
\end{proof}

There is ongoing interest in establishing connections between these parameters and the structure and behavior of the space \cite{Clax,JP,Kae,KMN,Sha,Wan,Yan,YW}. 

The weak parallelogram laws interact fruitfully with a generalized notion of orthogonality.  We say that two vectors $\mathbf{x}$ and $\mathbf{y}$  (in this order) in a Banach space $\mathscr{X}$ are orthogonal in the Birkhoff-James sense \cite{Jam}, and write $\mathbf{x} \perp_{\mathscr{X}} \mathbf{y}$, if
\[
    \| \mathbf{x} + t\mathbf{y} \| \geq \| \mathbf{x} \|
\]
for all scalars $t$.  A short exercise will show that this definition is equivalent to the traditional one when $\mathscr{X}$ is a Hilbert space.  More generally, however, it fails to be linear and symmetric. 
Weak parallelogram spaces enjoy a sort of Pythagorean Theorem with respect to Birkhoff-James orthogonality.  The following is from \cite{CR}.
\begin{Theorem}
Let $1 < r < \infty$. 
If  a smooth Banach space  $\mathscr{X}$ is $r$-LWP($C$),  then there exists a positive constant $K$ such that whenever $\mathbf{x} \perp_{\mathscr{X}} \mathbf{y}$ in $\mathscr{X}$, 
\begin{equation*}\label{LPyth}
  \| \mathbf{x} \|^r  +  K\|\mathbf{y}\|^r \leq \|\mathbf{x} + \mathbf{y} \|^r.
\end{equation*}
If $\mathscr{X}$ is $r$-UWP($C$),  then there exists a positive constant $K$ such that  whenever $\mathbf{x} \perp_{\mathscr{X}} \mathbf{y}$ in $\mathscr{X}$, 
\begin{equation*}\label{UPyth}
  \| \mathbf{x} \|^r  +  K\|\mathbf{y}\|^r \geq \|\mathbf{x} + \mathbf{y} \|^r
\end{equation*}
In either case, the constant $K$ can be chosen to be ${C}/{(2^{r-1}-1)}$.
\end{Theorem}

\section{The weak parallelogram constant $C_{p, r}$}\label{Crp}

Towards proving Theorem \ref{MAINT}, let us establish some basic facts about the weak parallelogram constant $C_{p, r}$ appearing in \eqref{wkparaeq2} and \eqref{wkparaeq5}. 
\begin{Lemma}\label{defcpqlem}
Suppose that $1 < p \leq 2$ and $2 \leq r \leq q$, where $1/p + 1/q=1$.  Then the constant $C_{p,r}$ defined by
\begin{equation}\label{cprdefeq}
       C_{p,r} := \inf_{0 \leq t < 1} \frac{2^{r-r/p}(1+t^p)^{r/p}-(1+t)^r}{(1-t)^r} 
\end{equation}
satisfies
$
     0 < C_{p,r} \leq 1.
$
\end{Lemma}
\begin{proof}  For any $t \in [0, 1)$, 
\begin{equation}\label{lp2dineq}
    (1+t)^r \leq 2^{r - r/p} (1 + t^p)^{r/p},  
\end{equation}
Indeed,  H\"{o}lder's inequality gives
   \begin{align*}
      (1+t)^r &=  (1\cdot 1 + t\cdot 1)^r \\
             &\leq  (1^p + t^p)^{r/p} (1^q + 1^q)^{r/q} \nonumber \\
             &=  2^{r - r/p} (1 + t^p)^{r/p},\nonumber
   \end{align*}
   and hence the numerator of the right side of  \eqref{cprdefeq} is nonnegative.  In fact, the condition for equality in \eqref{lp2dineq} will imply equality in our use of H\"{o}lder's inequality. This would mean that the vectors $(1, 1)$ and $(1, t)$ are linearly dependent -- which only holds when $t = 1$.  The denominator is obviously positive as well, and hence the fraction itself is positive for all $t \in [0, 1)$.
   
Let us look at the case $r=2$.  Two applications of L'H\^{o}pital's Rule yield 
$$\lim_{t\rightarrow 1-} \frac{2^{2-2/p}(1+t^p)^{2/p}-(1+t)^2}{(1-t)^2} = (2-p)+2(p-1) -1 > 0.$$
Thus the infimum in (\ref{cprdefeq}) must be positive when $r=2$. When $r=2$, the infimum in \eqref{cprdefeq} occurs in the limit as $t$ increases to 1.  Bynum and Drew \cite{BD} derived the value of this limit, with the result
$C_{p,2} = p-1$ when $1<p\leq 2$ (see Figure \ref{Four}).
\begin{figure}
 \includegraphics[width=.6\textwidth]{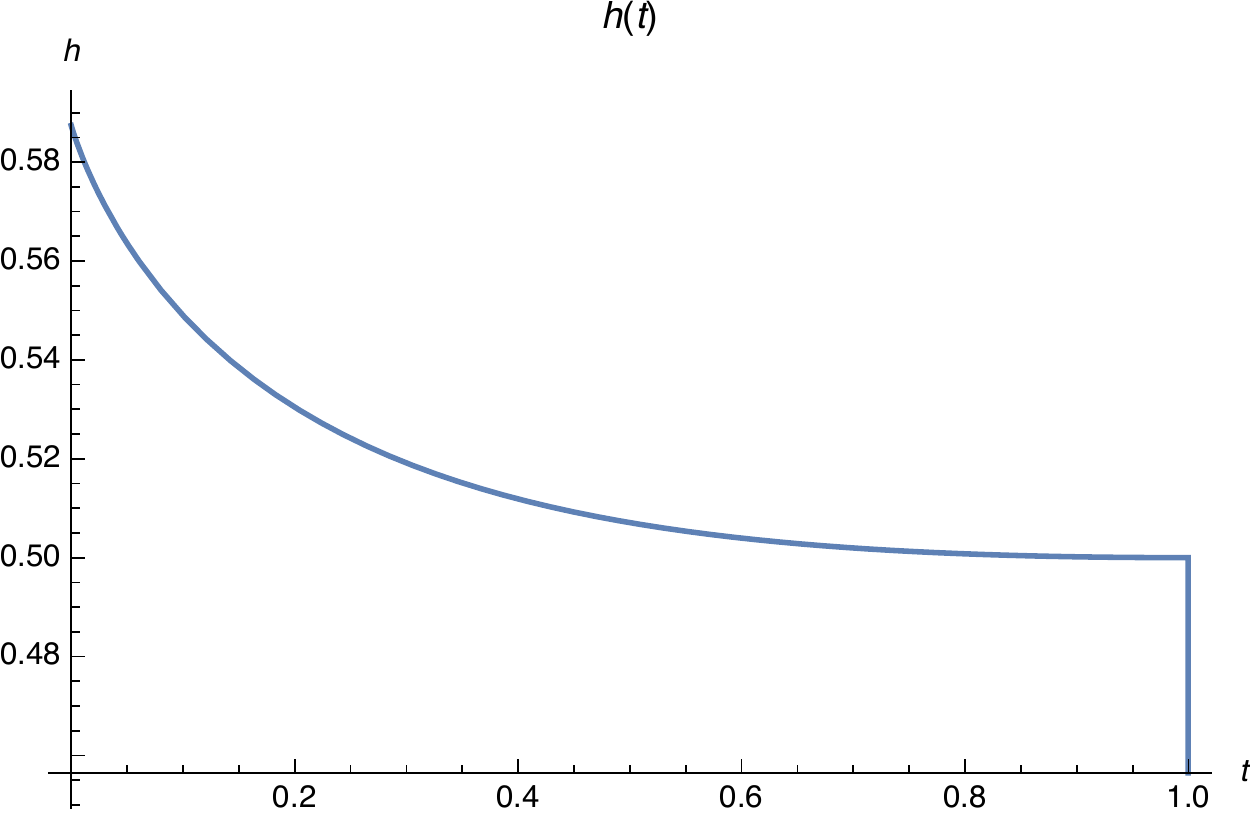}
 \caption{The graph of $h(t)$ from \eqref{hoft} with $p = 3/2$ and $r = 2$. Observe how this says that $C_{3/2, 2} = 1/2$. } 
 \label{Four}
\end{figure}

Next, for $r \in [2,q]$, $1<p\leq 2$, and $0\leq t <1$, consider
\begin{align*}
       & \frac{d}{dr} \left[\frac{2^{r-r/p}(1+t^p)^{r/p}-(1+t)^r}{(1-t)^r} \right]\\
       &=  \left\{(1-t)^r\Big[ 2^{r-r/p}(1+t^p)^{r/p}\log [2^{1-1/p}(1+t^{p})^{1/p}] - (1+t)^r\log(1+t)  \Big]\right.\\
       &\quad - \left.\Big[ \big( 2^{r-r/p}(1+t^p)^{r/p}-(1+t)^r\big)(1-t)^r \log(1-t) \Big]\right\}/(1-t)^{2r}.
\end{align*}
The final expression is nonnegative.  To see this, we use the fact that the function $\psi(s) = s^r\log s$ is increasing on interval $[1,\infty)$, the negativity of the factor $\log(1-t)$, and the previously derived inequality (\ref{lp2dineq}).   This shows that the expression
\[
       \frac{2^{r-r/p}(1+t^p)^{r/p}-(1+t)^r}{(1-t)^r}
\]
from the defining equation (\ref{cprdefeq}) is a nondecreasing function of $r$ for each $p$ and $t$.  It follows that 
\[
   \varliminf_{t\rightarrow 1-}  \frac{2^{r-r/p}(1+t^p)^{r/p}-(1+t)^r}{(1-t)^r} \geq (p-1) >0.
\]
 This forces $C_{p,r}$ to be positive for all parameter values.
 Finally, by substituting $t=0$ we find that
 \begin{align*}
     C_{p,r} &= \inf_{0 \leq t < 1} \frac{2^{r-r/p}(1+t^p)^{r/p}-(1+t)^r}{(1-t)^r} 
                 \leq  2^{r-r/p}-1 
                 \leq 1. \qedhere
\end{align*}
\end{proof}

On the other hand, when $2<r\leq q$, the value of $C_{p,r}$ is attained in (\ref{cprdefeq}) at some interior point $t$ of $[0,1)$.  To see this, note that a calculation will show that 
\begin{align*}
     & \quad \frac{d}{dt} \left[\frac{2^{r-r/p}(1+t^p)^{r/p}-(1+t)^r}{(1-t)^r}\right] \\
     & = \frac{ r(1-t)^{r-1} [2^{r/q}(1+t^p)^{r/p-1}(1+t^{p-1})-2(1+t)^{r-1}] }{(1-t)^{2r}}.
\end{align*}
When $t=0$, this takes the negative value $r(2^{r/q}-2)$.   On the other hand, if we write $t = 1-\epsilon$, and consider $\epsilon$ decreasing to zero, we obtain
$$\frac{2^{r-r/p}(1+t^p)^{r/p}-(1+t)^r}{(1-t)^r}  =  \frac{2^{r-r/p}(1+[1-\epsilon]^p)^{r/p}-(2-\epsilon)^r}{\epsilon^r}.$$
Written as a power series in $\epsilon$ the above becomes 
$$\left\{ 2^r \Big(\frac{r(p-1)}{8}\Big)\epsilon^2 + O(\epsilon^3) \right\}\epsilon^{-r},
$$
which, since $r > 2$, diverges to $+\infty$ as $\epsilon  \to 0$.

We are unable to present a closed form expression for $C_{p,r}$, though we can obtain some numerical estimates (see the examples below), and instead offer the following simple bounds.
\begin{Corollary}\label{ppsdpfoidsf}
   If $1<p<2<r\leq q$, where $1/p+1/q=1$, then 
   \[
         (p-1)^{r/2} \leq C_{p,r} < 2^{r/q} -1.
   \]
\end{Corollary}
\begin{proof}
  The first inequality comes from \cite[Thm.~2.1]{CR} (and the fact, as we will prove in a moment, that $C_{p, r}$ is the optimal weak parallelogram constant).  The second follows from comparing the infimum in (\ref{cprdefeq}) to the value at the endpoint $t=0$. 
\end{proof}

\section{Concrete examples}

Let demonstrate Theorem \ref{MAINT} with some specific examples. To find $C_{p, r}$, optimal constant, we see from Lemma \ref{defcpqlem} the need to minimize the function 
\begin{equation}\label{hoft}
h(t)  = \frac{2^{r-r/p}(1+t^p)^{r/p}-(1+t)^r}{(1-t)^r}, \quad 0 \leq t < 1.
\end{equation}

We first examine some values of $1 < p < 2$. 

\begin{Example}
Let $p = 3/2$ (for which $q = 3$) and $r = 5/2$ (note that $2 \leq r \leq q$). In this case 
$$h(t)= \frac{2^{5/6} \left(t^{3/2}+1\right)^{5/3}-(t+1)^{5/2}}{(1-t)^{5/2}}.$$ One can see from the graph in Figure \ref{One} that $h$ has a minimum  between $t = 0.02$ and $t = 0.03$. By Newton's method (finding the root of the zero of $h'(t)$),
the minimum occurs at $t \approx .027307$ and so 
$C_{3/2, 5/2} \approx 0.777545.$
The bounds from Corollary \ref{ppsdpfoidsf} give the estimate 
$0.420448 \leq C_{3/2, 5/2} \leq 0.922331.$
\begin{figure}
 \includegraphics[width=.6\textwidth]{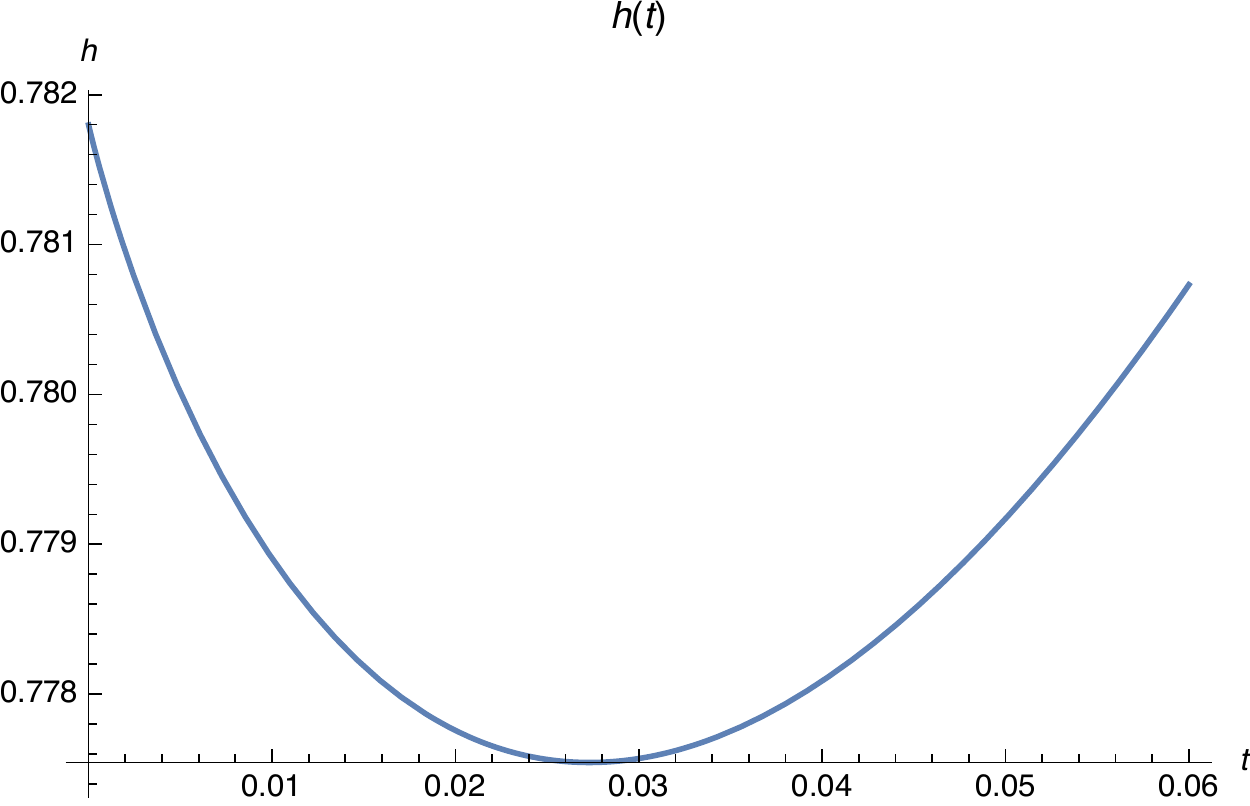}
 \caption{The graph of $h(t)$ from \eqref{hoft} with $p = 3/2$ and $r = 5/2$.}
 \label{One}
\end{figure}
\end{Example}

\begin{Example}
As in the previous example, let $p = 7/4$ (for which $q = 7/3$) and $r = 22/10$ (note $2 \leq r \leq q$). In this case 
$$h(t)= \frac{2^{33/35} \left(t^{7/4}+1\right)^{44/35}-(t+1)^{11/5}}{(1-t)^{11/5}}.$$ One can see from the graph in Figure \ref{Two} that $h$ has a minimum  between $t = 0.03$ and $t = 0.05$. Similarly, as done in the previous example, the minimum occurs at $t \approx 0.0402087$ and so 
$C_{7/4, 22/10} \approx 0.919875.$
The bounds from Corollary \ref{ppsdpfoidsf} yield the estimate 
$0.728731 \leq C_{7/4, 22/10} \leq 0.922331.$
\begin{figure}
 \includegraphics[width=.6\textwidth]{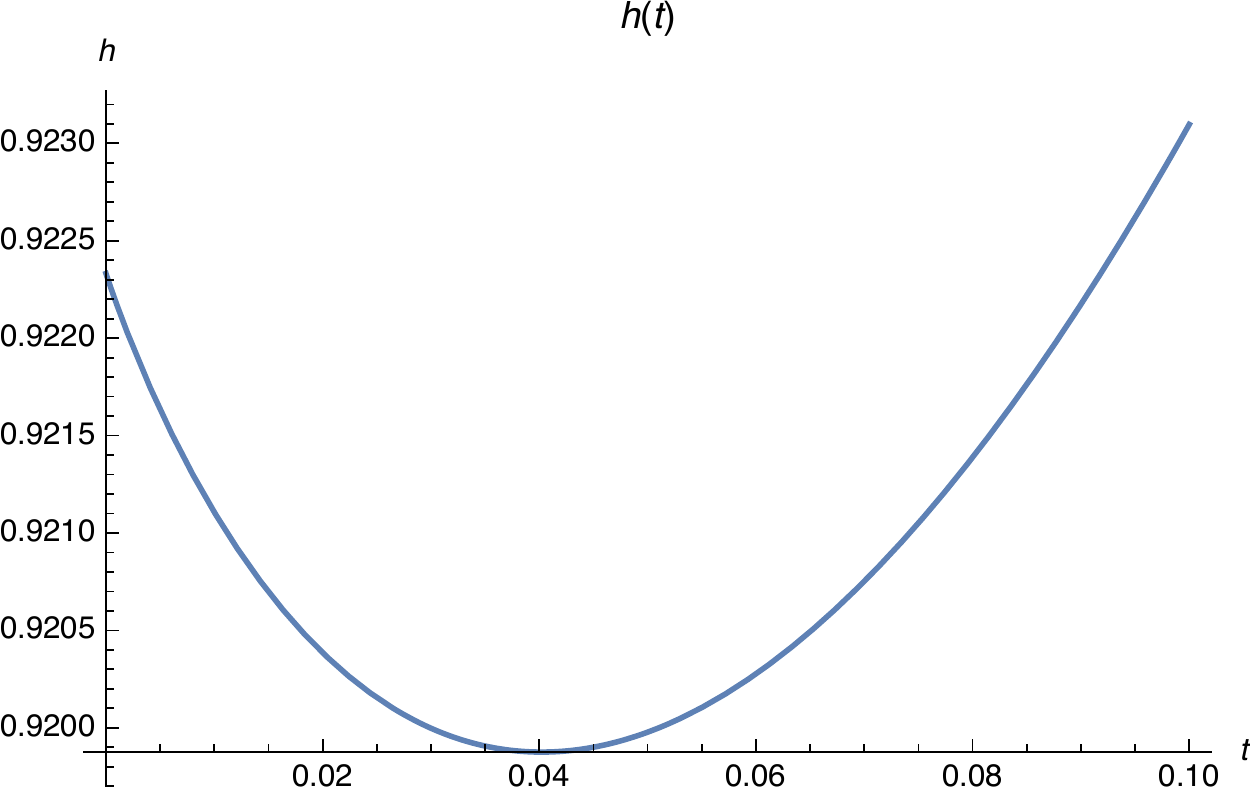}
 \caption{The graph of $h(t)$ from \eqref{hoft} with $p = 7/4$ and $r = 22/10$.}
 \label{Two}
\end{figure}
\end{Example}

Next we examine the optimal constant in Theorem \ref{MAINT} when $p > 2$. 

\begin{Example}
Let 
$p = 5/2$ (for which $q = 5/3$) and $r = 21/12$ (for which $r' = 7/3$). By Theorem \ref{MAINT}, the optimal constant is 
$C_{q, r'}^{-p/q} = C_{5/3, 7/3}^{-3/2}.$
To compute $C_{5/3, 7/3}$ we need to minimize 
$$h(t) = \frac{2^{14/15} \left(t^{5/3}+1\right)^{7/5}-(t+1)^{7/3}}{(1-t)^{7/3}}, \quad 0 \leq t < 1.$$
Using our previous examples (see Figure \ref{Three}) we get 
$C_{5/3, 7/3} \approx 0.908148$ and so 
$C_{5/3, 7/3}^{-3/2} \approx 1.15549.$
\begin{figure}
 \includegraphics[width=.6\textwidth]{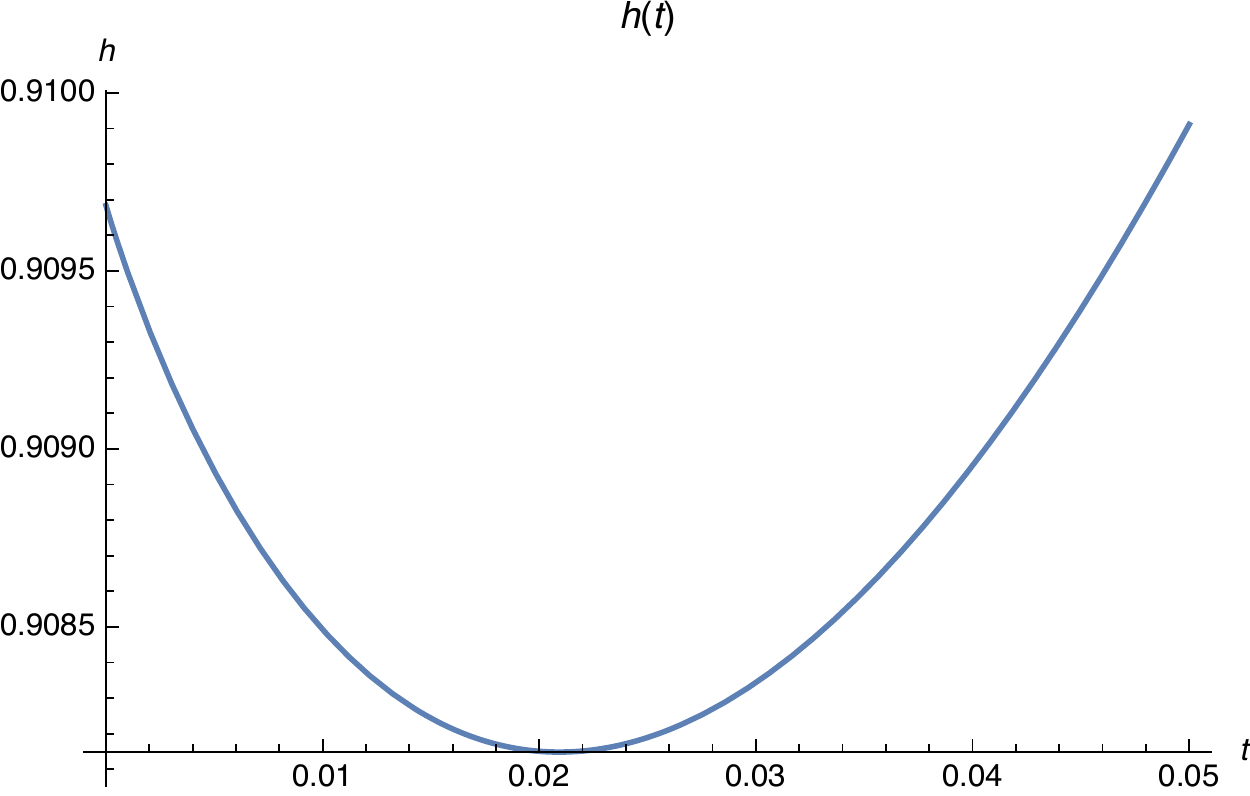}
 \caption{The graph of $h(t)$ from \eqref{hoft} with $p = 5/3$ and $r = 21/12$.}
 \label{Three}
\end{figure}
\end{Example}

\section{Weak Parallelogram Laws for $L^p$}

This section contains the proof of Theorem \ref{MAINT}.  We will focus on the parameter values $1<p\leq 2$ since, for $2 \leq p <\infty$, we can use the following duality property for weak parallelogram spaces \cite{CH1}.

\begin{Proposition}
\label{dualitythm}
Let $1 < p < \infty$,  $q$ be its conjugate index, and $C > 0$. If $\mathscr{X}$ is a Banach space and $\mathscr{X}^{*}$ its dual, we have the following:
\begin{enumerate}
\item[(i)] $\mathscr{X}$ is $p$-LWP($C$) if and only if $\mathscr{X}^*$ is $q$-UWP($C^{-q/p}$); 
\item[(ii)] $\mathscr{X}$ is $p$-UWP($C$) if and only if $\mathscr{X}^*$ is $q$-LWP($C^{-q/p}$).
\end{enumerate}
\end{Proposition}

%
%
%
Proceeding to the proof of Theorem \ref{MAINT}, the weak parallelogram constants appearing in \eqref{wkparaeq1} and \eqref{wkparaeq3} are unity, and hence must already be optimal \cite[Proposition 2.2]{CH1}.
For the case in (\ref{wkparaeq2}), the derivation of the optimal constant will proceed below in stages.  The remaining three cases are the corresponding dual statements, in the sense that ``lower'' and ``upper'' are interchanged, each exponent is switched with its H\"{o}lder conjugate, and the weak parallelogram constants do just the right thing via Proposition \ref{dualitythm}.  Here is the last bit of heavy lifting needed to prove (\ref{wkparaeq2}).

%
%
%

%
\begin{Lemma}\label{hposlem}
  Let $1 < p \leq 2 \leq r \leq q$, where $1/p+1/q =1$.  For any real numbers $u$ and $v$, we have 
  $
       |u+v|^r + C_{p,r} |u-v|^r \leq 2^{r-r/p}(|u|^p+|v|^p)^{r/p}.
  $
\end{Lemma}
\begin{proof}
For real values of $t$ define
   \begin{equation*}\label{hoftdefeq}
        k(t) := 2^{r-r/p} (1+|t|^p)^{r/p} - |1+t|^r - C_{p,r}|1-t|^r.
   \end{equation*}
We can readily confirm that 
$k(t) = |t|^r k(1/t), t \not = 0$, and $k(t) \geq k(|t|).$  Consequently, we will be done if we can show that $k(t) \geq 0$ for all $t \in [0,1]$.  But this follows from Lemma \ref{defcpqlem}.
\end{proof}

The next step in the proof of Theorem \ref{MAINT} requires Hanner's inequality \cite[Thm.~1]{MR0077087}:
  If $1 < p \leq 2$ and $F, G \in L^p$, then
\begin{equation}\label{HI}
  \big( \|F\|_p + \|G\|_p\big)^p + \big| \|F\|_p - \|G\|_p\big|^p
    \leq  \|F+G\|_p^p +  \|F-G\|_p^p.
    \end{equation}

The following establishes \eqref{wkparaeq2}. 

%
%
 %

\begin{Lemma}
   If $1 < p \leq 2 \leq r \leq q$, where $1/p + 1/q =1$, then $L^p$ is
$r$-LWP($C_{p,r}$). Moreover, the constant $C_{p, r}$ is optimal. 
\end{Lemma}
\begin{proof}
    For $f, g \in L^p$ let 
    $$2u := \| f + g \|_p
   + \| f - g \|_p, \quad 2v := \|f + g \|_p
   - \|  f - g \|_p.$$ By Lemma \ref{hposlem}, Hanner's Inequality (with $F := \tfrac{1}{2} (f + g)$ and $G := \tfrac{1}{2} (f - g)$ in \eqref{HI}) and H\"{o}lder's Inequality,
\begin{align*}
   & \| f + g \|_p^r + C_{p,r}\| f - g \|_p^r \\
   &= (u+v)^r + C_{p,r} (u-v)^r \\
   &\leq 2^{r-r/p}(|u|^p+|v|^p)^{r/p} \\
        &= 2^{r-r/p}\Big[\big(  \textstyle{\frac{1}{2}}\| f + g \|_p+  \textstyle{\frac{1}{2}}\| f - g \|_p\big)^p
             +\big| \textstyle{\frac{1}{2}}\| f + g \|_p- \textstyle{\frac{1}{2}} \| f - g \|_p\big|^p \Big]^{r/p}\\
        &\leq 2^{r-r/p}\big(\|f\|_p^p+\|g\|_p^p\big)^{r/p}\\
        &\leq 2^{r-r/p}\big(\|f\|_p^{p(r/p)}+\|g\|_p^{p(r/p)}\big)^{(r/p)(p/r)}\big(1^{r/(r-p)}+ 1^{r/(r-p)}  \big)^{(r/p)([r-p]/r)}\\
        &= 2^{r-r/p}\big(\|f\|_p^{r}+\|g\|_p^{r}\big) 2^{r/p -1}\\
        &= 2^{r-1}\big( \|f\|_p^{r}+\|g\|_p^{r}\big).
\end{align*}
This verifies that $C_{p,r}$ is sufficiently small so that $L^p$ satisfies $r$-LWP($C_{p,r}$).  

To see that $C_{p, r}$ is optimal, let us first consider the special case that the space is $\ell^p$, with the natural basic sequence
$\mathbf{e}_0 := (1, 0, 0,\ldots)$, $\mathbf{e}_1 := (0, 1, 0, 0,\ldots)$, and so on.
Define $\mathbf{a} = t\mathbf{e}_0 + \mathbf{e}_1$ and $\mathbf{b} = \mathbf{e}_0 + t\mathbf{e}_1$ for some $t \in [0, 1)$.  Then
 \begin{align*}
       \frac{2^{r-1}\big(\|\mathbf{a}\|_p^r +\|\mathbf{b}\|_p^r\big) - \|\mathbf{a}+\mathbf{b}\|_p^r}{\|\mathbf{a}+\mathbf{b}\|_p^r}
           &= \frac{2^{r-1}2(1+t^p)^{r/p} - 2^{r/p}(1+t)^r}{2^{r/p}(1-t)^r}\\
           &= \frac{2^{r-r/p}(1+t^p)^{r/p} - (1+t)^r}{(1-t)^r}.
 \end{align*}
The infimum (for $t \in [0, 1)$) of the final expression is $C_{p,r}$.  Thus there is a sequence of values of $t$ for which equality to the infimum holds in the limit. Consequently,
the lower weak parallelogram law
\[
      \|  \mathbf{a} + \mathbf{b} \|_p^r + C\|  \mathbf{a} - \mathbf{b} \|_p^r \leq 2^{r-1}\big( \|\mathbf{a}\|_p^{r}+\|\mathbf{b}\|_p^{r}\big)
\]
cannot be valid for any value of $C$ greater than $C_{p,r}$.  This treats the special case of $L^p$ being the sequence space $\ell^p$.

Finally, if the general $L^p$ space has at least two disjoint measurable sets of non-zero measure, then there is an obvious  linear isomorphism between the span of $\{\mathbf{e}_0, \mathbf{e}_1\}$ in $\ell^p$, and the subspace of $L^p$ generated by the two measurable sets.  Then the calculation of the previous paragraph once again establishes that $C_{p,r}$ is optimal.
\end{proof}



%
%
%

\bibliographystyle{plain}

\bibliography{references3}

\end{document}